\documentclass[12pt, a4paper, leqno]{amsart}

\usepackage{amsmath}
\usepackage{amsfonts}
\usepackage{amssymb}
\usepackage{color}
\usepackage[T1]{fontenc} 
\usepackage{url}

\numberwithin{equation}{section}

\newtheorem{theorem}{Theorem}[section]
\newtheorem{definition}[theorem]{Definition}
\newtheorem{lemma}[theorem]{Lemma}

\newtheorem{corollary}[theorem]{Corollary}

\theoremstyle{remark}
\newtheorem{remark}[theorem]{Remark}
\newtheorem{example}[theorem]{Example}

\theoremstyle{plain}

\newcommand{\R}{\mathbb{R}}
\newcommand{\N}{\mathbb{N}}

\newcommand{\cS}{\mathcal{S}}

\newcommand{\Rn}{{\mathbb{R}^n}}
\newcommand{\Zn}{{\mathbb{Z}^n}}

\DeclareMathOperator{\supp}{supp}

\newcommand{\Lh}{L^{p,\lambda}}
\newcommand{\Lnh}{\mathcal{L}^{p,\lambda}}
\newcommand{\Vz}{V_0L^{p,\lambda}}
\newcommand{\Vi}{V_\infty L^{p,\lambda}}
\newcommand{\Vast}{V^{(\ast)} L^{p,\lambda}}

\newcommand{\Vziast}{V_{0,\infty}^{(\ast)}L^{p,\lambda}}
\newcommand{\Zor}{\mathbb{L}^{p,\lambda}}

\begin{document}

\title{Approximation in Morrey spaces}

\author[A. Almeida]{Alexandre Almeida$^{*}$}
\address{Center for R\&D in Mathematics and Applications, Department of Mathematics, University of Aveiro, 3810-193 Aveiro, Portugal}
\email{jaralmeida@ua.pt}

\author[S. Samko]{Stefan Samko}
\address{University of Algarve, Department of Mathematics, Campus de Gambelas, 8005-139 Faro, Portugal}
\email{ssamko@ualg.pt}

\thanks{$^*$ Corresponding author.}
\thanks{This work was supported in part by the Portuguese Foundation for Science and Technology (FCT -- Funda\c{c}\~{a}o para a Ci\^{e}ncia e a Tecnologia), through CIDMA -- Center for Research and Development in Mathematics and Applications, within project UID/MAT/04106/2013.}

\date{\today}



\begin{abstract}
A new subspace of Morrey spaces whose elements can be approximated by infinitely differentiable compactly supported functions is introduced. Consequently, we give an explicit description of the closure of the set of such functions in Morrey spaces. A generalization of known embeddings of Morrey spaces into weighted Lesbesgue spaces is also obtained.
\end{abstract}

\maketitle


\section{Introduction}\label{sec:intro}

Morrey spaces are widely used in applications to regularity properties of solutions to PDE including the study of Navier-Stokes equations (see \cite{Tri13} and references therein). Although such spaces allow to describe local properties of functions better than Lebesgue spaces, they have some unpleasant issues. It is well known that Morrey spaces are non separable and that the usual classes of nice functions are not dense in such spaces.

The theory of Morrey spaces goes back to Morrey \cite{Mor38} who considered related integral inequalities in connection with regularity properties of solutions to nonlinear elliptic equations. In the form of Banach spaces of functions, called thereafter Morrey spaces, the ideas of Morrey \cite{Mor38} were further developed by Campanato \cite{Cam64}. A more systematic study of these (and even more general) spaces was carried out by Peetre \cite{Pee69} and Brudnyi \cite{Bru71}. We refer to the books \cite{Adam15,Gia83,Pick13,Tay00,Tri13} and  the overview \cite{RafNSamSam13} for additional references and basic properties and generalizations of Morrey spaces. We also refer to \cite{AdamXiao12} for Harmonic Analysis in Morrey spaces and the special issue with Editorial \cite{SawGGulTan14} for a discussion on related function spaces.

In \cite{Zor86} it was observed that the set of functions in Morrey spaces for which the translation is continuous in Morrey norm plays an important role in approximation. This was sketched in \cite[Proposition.~3]{Zor86} and also discussed in \cite{Chia92} and \cite{Kato92}. A description of this set of functions in easily verified terms  seems to be a difficult task for unbounded domains; at least the authors were unable to find any one in the literature. In \cite{Zor86} and \cite{Kato92} there were given some results on equivalence between belonging to Zorko space and approximation by mollifiers.

In this paper we introduce some subspaces of Morrey spaces by adding some ``vanishing'' type conditions. One of them is similar to the already known vanishing property, but related to the behavior at infinity instead of that at the origin. Another is connected with the truncation of functions to the exterior of large balls. These two additional properties, together with the vanishing property at the origin, allow us to show that all elements in this new subspace, denoted in the sequel by $\Vziast$, may be approximated by $C_0^\infty$ functions in Morrey norm. In particular, dilation type identity approximations with integrable kernels strongly converges for all the functions in $\Vziast$.

As we shall see, the set $\Vziast$ is strictly smaller than Zorko class. An example of a Morrey function belonging to Zorko class but not belonging to this new subspace is given below. Moreover, $\Vziast$ is closed in the Morrey space. Consequently, we give an explicit description of the closure of $C_0^\infty$ in Morrey spaces. This closure plays an important role in Harmonic Analysis on Morrey spaces, including Calder\'{o}n-Zygmund theory, since its dual constitutes a predual of Morrey spaces (cf. \cite{AdamXiao12}). Preduals of Morrey spaces have been studied by many authors, see the recent book \cite{Adam15} and the paper \cite{RosTri15} for further details and references.

%

In addition to consideration of smaller subspaces of Morrey spaces, we also generalize some known embeddings of these space into weighted Lebesgue spaces. Various examples are presented showing the difference between all those spaces.

The paper is organized as follows. After some notation and preliminaries on Morrey spaces, the main ideas and results are given in Sections \ref{sec:subspaces}, \ref{sec:strictembed}, \ref{sec:approx-subspaces} and \ref{sec:improved-embed}. In Section~\ref{sec:subspaces} we introduce new Morrey subspaces and discuss some of their properties, including the invariance with respect to convolutions with integrable kernels. The relation between these subspaces and already known classes is discussed in Section~\ref{sec:strictembed}. In Section~\ref{sec:approx-subspaces} we study the approximation by nice functions in different Morrey subspaces. The main results here concern the approximation of functions in $\Vziast$ by $C_0^\infty$ functions in Morrey norm, and  the description of the closure of $C_0^\infty$ in Morrey spaces in an explicit way. Finally, in Section~\ref{sec:improved-embed}, we generalize some known embeddings of Morrey spaces into weighted $L^p$ spaces.


\section{Preliminaries on Morrey spaces}\label{sec:prelim}

We denote by $\mathbb{R}^{n}$ the $n$-dimensional real
Euclidean space. We write $B(x,r)$ for the open ball in
$\mathbb{R}^{n}$ centered at $x\in \mathbb{R}^{n}$ with radius $r>0$.
If $E\subseteq {\mathbb{R}^{n}}$ is a  measurable set, then $|E|$ stands
for its (Lebesgue) measure and $\chi_{E}$ denotes its characteristic
function. By $\supp f$ we denote the support of the function $f$. The notation $X\hookrightarrow Y$ stands for continuous embeddings
from the normed space $X$ into the normed space $Y$.
We use $c$ as a generic positive constant, i.e., a constant whose
value may change with each appearance. The expression $f
\lesssim g$ means that $f\leq c\,g$ for some independent constant
$c$, and $f\approx g$ means $f \lesssim g \lesssim f$.


The set $\cS(\Rn)$ denotes the usual Schwartz class of complex-valued rapidly decreasing infinitely differentiable
functions on $\Rn$ and $\cS'(\Rn)$ is the set of all tempered distributions in $\Rn$. The class $C_0^\infty(\Rn)$ consists of all those functions in $\cS(\Rn)$ with compact support. As usual, $L^p(E)$ with $1\leq p <\infty$ and $E$ a measurable subset of $\Rn$, is the standard Lebesgue space normed by
$$
\|f\|_{L^p(E)}= \left( \int_E |f(x)|^p\,dx\right)^{1/p}.
$$
When $E=\Rn$ we shall write only $\|\cdot\|_p$ instead of $\|\cdot\|_{L^p(\Rn)}$.


In the sequel $\Omega \subseteq \Rn$ is an open set, $\widetilde{B}(x,r):= B(x,r) \cap \Omega$ for $x\in\Omega$, $r>0$, and $L^p_{\textrm{loc}}(\Omega)$ stands for the space of all locally $p$-integrable functions on $\Omega$. Moreover, we use the notation
$$
\mathfrak{M}_{p,\lambda}(f;x,r):= \frac{1}{r^\lambda} \int_{\widetilde{B}(x,r)} |f(y)|^p\,dy\,, \ \ \ \ \ \ x\in\Omega, \ \ \ r>0, \ \ \ f\in L^1_{\textrm{loc}}(\Omega)
$$
(as in \cite{NSam13}).

\subsection{Classical Morrey spaces}


Let $1\leq p < \infty$ and $0\leq \lambda \leq n$. The \emph{homogeneous Morrey space} $L^{p,\lambda}(\Omega)$ (sometimes also called global Morrey space) is defined as
$$
L^{p,\lambda}(\Omega)=\left\{ f\in L^p_{\textrm{loc}}(\Omega): \sup_{x\in\Omega,\; r>0} \mathfrak{M}_{p,\lambda}(f;x,r) < \infty\right\}.
$$
Its \emph{inhomogeneous} counterpart $\mathcal{L}^{p,\lambda}(\Omega)$ (sometimes also called local Morrey space) is the space
$$
\mathcal{L}^{p,\lambda}(\Omega)=\left\{ f\in L^p_{\textrm{loc}}(\Omega): \sup_{x\in\Omega,\; r\in]0,1]} \mathfrak{M}_{p,\lambda}(f;x,r) < \infty\right\}.
$$
They are both Banach spaces equipped with the corresponding norms
\begin{equation}\label{HomMorreynorm}
\|f\|_{L^{p,\lambda}(\Omega)}:= \sup_{x\in\Omega,\; r>0} \mathfrak{M}_{p,\lambda}(f;x,r)^{1/p}
\end{equation}
and
\begin{equation}
\|f\|_{\mathcal{L}^{p,\lambda}(\Omega)}:= \sup_{x\in\Omega,\; r\in]0,1]} \mathfrak{M}_{p,\lambda}(f;x,r)^{1/p}.
\end{equation}

In general, we write $\|\cdot\|_{p,\lambda}$ for short to denote the Morrey norms above when it is clear which version is being considered in each situation. Moreover, we shall omit the reference to the domain when it coincides with the whole $\Rn$; e.g., we write $\Lh$ instead of $\Lh(\Rn)$ for short.

Morrey spaces can be seen as a refinement of $L^p$ spaces and they are part of a larger scale, called Morrey-Campanato spaces, which also includes Hölder spaces and the space $BMO$.

It can be shown that both spaces $L^{p,\lambda}(\Omega)$ and $\mathcal{L}^{p,\lambda}(\Omega)$ would reduce to $\{0\}$ if $\lambda>n$. Moreover, we have
$$
L^{p,0}(\Omega) = L^{p}(\Omega) \,, \ \ \ \ \ \ L^{p,n}(\Omega) = L^{\infty}(\Omega) = \mathcal{L}^{p,n}(\Omega)  \ \ \ \ \text{and} \ \ \ \ \mathcal{L}^{p,0}(\Omega) = \mathcal{L}^{p}(\Omega),
$$
where $\mathcal{L}^{p}(\Omega)$ stands for the \emph{uniform Lebesgue space} normed by
\begin{equation}\label{Lp-uniform}
\|f\|_{\mathcal{L}^{p}(\Omega)}:= \sup_{x\in\Omega} \|f\|_{L^p(\widetilde{B}(x,1))}.
\end{equation}
It is clear that we always have
$$L^{p,\lambda}(\Omega) \hookrightarrow \mathcal{L}^{p,\lambda}(\Omega) \ \ \ \ \ \ \text{and} \ \ \ \ \ \ L^{\infty}(\Omega) \hookrightarrow \mathcal{L}^{p,\lambda}(\Omega).$$

Given $1\leq p \leq q < \infty$ and $0 \leq \lambda,\mu\leq n$, an application of Hölder's inequality yields the embedding
$$
L^{q,\mu}(\Omega) \hookrightarrow \Lh(\Omega)
$$
under the condition
$$
\frac{\lambda-n}{p} = \frac{\mu-n}{q}
$$
if $|\Omega|=\infty$, and the condition
$$
\frac{\lambda-n}{p} \leq \frac{\mu-n}{q}
$$
if $|\Omega|<\infty$. In particular, for bounded domains $\Omega$, we have
$$
L^{\infty}(\Omega) \hookrightarrow \Lh(\Omega) \hookrightarrow L^p(\Omega).
$$

Easy calculations show that the homogeneous Morrey norm has the property
$$
\|f(t\cdot)\|_{L^{p,\lambda}(\Omega)} = t^{\frac{\lambda-n}{p}}\,\|f\|_{L^{p,\lambda}(\Omega)}\,, \ \ \ \ t>0,
$$
which reveals the homogeneous nature of the spaces $L^{p,\lambda}(\Omega)$. It is not hard to check that a Young's convolution inequality holds true for homogeneous Morrey spaces $\Lh(\Rn)$:
\begin{equation}\label{Young}
\|f\ast g\|_{p,\lambda} \leq \|g\|_1\,\|f\|_{p,\lambda} \,, \ \ \ \ 0\leq\lambda\leq n, \ \ \ 1\leq p<\infty.
\end{equation}

The following chain of embeddings is known (cf. \cite[Theorem~3.1]{RosTri15}):


$$
\cS \hookrightarrow B^{\frac{\lambda}{p}}_{p,p} \hookrightarrow L^{\frac{pn}{n-\lambda}} \hookrightarrow L^{p,\lambda} \hookrightarrow \mathcal{L}^{p,\lambda} \hookrightarrow L^p(w_\alpha) \hookrightarrow \cS'\,,
$$
with $1<p<\infty$, $0<\lambda<n$ and $\alpha<-n/p$, where $B^{s}_{q,r}$ stands for classical Besov spaces on $\Rn$ (see, e.g., \cite{Tri83}), and $L^p(w_\alpha)$ denotes the weighted Lebesgue space on $\Rn$ normed by
$$
\|f\|_{L^p(w_\alpha)}:= \|w_\alpha f\|_p\,,
$$
with the notation $w_\alpha(x):=(1+|x|^2)^{\alpha/2}$. All these embeddings are strict. For example,
\begin{equation}\label{ex:homogeneous}
|x|^{-\frac{n-\lambda}{p}} \in L^{p,\lambda} \ \ \ \text{but} \ \ \ \ |x|^{-\frac{n-\lambda}{p}} \notin L^{\frac{pn}{n-\lambda}}.
\end{equation}
As regards the embeddings of Morrey spaces into weighted Lesbesgue spaces, the weight above can be improved when we consider the homogeneous space $L^{p,\lambda}$. If fact, it also holds
$$\Lh \hookrightarrow L^p(w_\alpha) \ \ \ \ \ \text{if} \ \ \ \alpha<-\lambda/p$$
(cf. \cite[p.132]{Kato92}). Later, in Section \ref{sec:improved-embed}, we shall return to such embeddings into weighted Lebesgue spaces and generalize the corresponding results given in \cite{Kato92} and \cite{RosTri15}.

Morrey spaces $L^{p,\lambda}$ and $\mathcal{L}^{p,\lambda}$ are non separable when $\lambda>0$. Moreover, none of the usual classes of smooth functions (e.g., $C^\infty_0$ or  $\cS$) is dense in  $\Lh$ and $\Lnh$, see \cite[Proposition~3.7]{RosTri15}. Nevertheless, $C^\infty_0$ is dense in the dual $\big(\overset{\circ}{L}{}^{p,\lambda}\big)^\prime$, where $\overset{\circ}{L}{}^{p,\lambda}$ denotes the closure of $C^\infty_0$ in Morrey norm. This fact, observed in \cite{AdamXiao12}, plays a crucial role in the study of Harmonic Analysis on Morrey spaces. One has
\begin{equation}\label{bidual}
\big(\overset{\circ}{L}{}^{p,\lambda}\big)^{\prime\prime} = \Big(\big(\overset{\circ}{L}{}^{p,\lambda}\big)^\prime\Big)^\prime = L^{p,\lambda}\,, \ \ \ \ \ 1<p<\infty, \ \ \ 0\leq \lambda < n
\end{equation}
(cf. \cite[Corollary~4.2 and Remark~4.3]{RosTri15}), which extends to Morrey spaces the well known assertion (for Lebesgue spaces) $\big(\overset{\circ}{L}{}^{p,0}\big)^{\prime\prime} = \big(L^{p}\big)^{\prime\prime} = L^{p}$. Preduals of Morrey spaces have been investigated by many researchers starting in the paper \cite{Zor86}. In \cite{AdamXiao12} it was described by means of Hausdorff capacities and Muckenhoupt weights. We refer to \cite{RosTri15} for a study in the framework of tempered distributions and further discussions including historical remarks.

\subsection{Known subspaces}

Approximations to the identity do not behave well in general Morrey spaces, since these spaces may contain functions with singularities like homogeneous functions (cf. \eqref{ex:homogeneous}). This is  one reason supporting the need of finding appropriate subspaces of Morrey spaces were certain nice properties still hold.

The so-called \emph{vanishing Morrey space} (at the origin) $\Vz(\Omega)$ is the subspace of $\Lh(\Omega)$ consisting of all those functions $f$ such that
\begin{equation}\label{vzero}
\lim_{r\to 0}\, \sup_{x\in\Omega}\, \mathfrak{M}_{p,\lambda}(f;x,r) =0. \tag{$V^0$}
\end{equation}
This is a closed set of $\Lh$ (cf. Lemma \ref{lem:closed} below) and it was introduced by Chiarenza and Franciosi \cite{Chia92} in the study of elliptic equations in non divergence form; see also Vitanza \cite{Vit90, Vit93} for regularity results  for elliptic equations with coefficients in such subspace. Note that
$$
\lim_{r\to 0}\, \sup_{x\in\Omega}\, \mathfrak{M}_{p,\lambda}(f;x,r) =0 \ \ \ \Longleftrightarrow \ \ \ \lim_{r\to 0} \sup_{x\in\Omega, \atop t\in]0,r]} \mathfrak{M}_{p,\lambda}(f;x,t) =0.
$$
Since all $L^p$ functions have this vanishing property, we have $V_0L^{p,0}(\Omega) = L^{p}(\Omega)$. However, if $\lambda>0$, then $\Vz(\Omega)$ is a proper subset of $\Lh(\Omega)$ (cf. Example \ref{ex:notinV0Vinfty}).

An interesting feature of $\Vz(\Omega)$, not shared by the whole Morrey space $\Lh(\Omega)$, is the possibility of approximation by bounded functions on bounded domains, as observed in \cite[Lemma~1.1]{Chia92}.

Another important subset of $\Lh(\Omega)$ is the so-called \emph{Zorko subspace} $\Zor(\Omega)$, introduced in \cite{Zor86}, consisting of all functions $f\in\Lh(\Omega)$ on which the translation operator is continuous, i.e.,
\begin{equation}\label{zorko-subs}
\Zor(\Omega):=\{ f\in\Lh(\Omega): \ \ \|\tau_\xi f-f\|_{p,\lambda} \rightarrow 0 \ \ \ \text{as} \ \ \ \xi\to 0\},
\end{equation}
where $\tau_\xi f = f(\cdot-\xi)$, $\,\xi\in\Rn$  (setting $f(x)=0$ for $x\in\Rn\setminus\Omega$). This is also a closed set in Morrey spaces and it is known that approximation via mollifiers is possible in such subspace, see \cite[Proposition 3]{Zor86}.

In the limiting case $\lambda=0$ we have $\mathbb{L}^{p,0}(\Omega) = V_0L^{p,0}(\Omega) = L^p(\Omega)$. For $\lambda>0$, in general the previous subspaces are related as follows:
\begin{equation}\label{zorvzmor}
\Zor(\Omega) \subset \Vz(\Omega) \subset \Lh(\Omega)
\end{equation}
(cf. \cite[Corollary 3.3]{Kato92}). Nevertheless, taking into account \cite[Lemma 1.2]{Chia92}, we have
$$
\Zor(\Omega) = \Vz(\Omega) \ \ \ \ \text{if} \ \ \ \ |\Omega|<\infty.
$$
For unbounded domains, e.g., $\Omega=\Rn$, the inclusions in \eqref{zorvzmor} are strict.


\section{New subspaces of Morrey spaces}\label{sec:subspaces}

In the sequel we introduce new subspaces of Morrey spaces by means of vanishing type conditions at infinity and show how they are related with the already known subspaces mentioned above. The context now is the whole $\Rn$, so that we omit the domain in the notations as agreed before.


\begin{definition}\label{def:vinfty}
Let $0\leq \lambda < n$ and $1\leq p < \infty$. We define $\Vi$ as the subset of $\Lh$ consisting of all those functions $f$ such that
\begin{equation}
\lim_{r\to \infty}\, \sup_{x\in\Rn}\, \mathfrak{M}_{p,\lambda}(f;x,r) =0. \tag{$V^\infty$}
\end{equation}
\end{definition}
Note that the subset $\Vi$ is defined similarly to $\Vz$, but now in terms of a corresponding vanishing property at infinity.

\begin{example}\label{ex:notinV0Vinfty}
For $\alpha,\beta>0$ let
$$
f_{\alpha,\beta}(x):= \left\{
\begin{array}{ll}
  \frac{1}{|x|^\alpha}\,, & |x|\leq 1, \\
  \frac{1}{|x|^\beta}\,, & |x| > 1.
\end{array}
\right.
$$
Then
$$f_{\alpha,\beta}\in \Vz\cap \Vi \ \ \ \ \text{if} \ \ \ \ \alpha<\frac{n-\lambda}{p}<\beta.$$
In the limiting case $\alpha=\frac{n-\lambda}{p}=\beta$, we have
$$f_{\alpha,\beta} \in \Lh \ \ \ \ \text{but} \ \ \ \ f_{\alpha,\beta} \notin \Vz \cup \Vi.$$
\end{example}

We introduce another subset by distinguishing a vanishing property after truncation in large balls. For convenience we use the notation
$$
\mathcal{A}_{N,p}(f) := \sup_{x\in\Rn}\,\int_{B(x,1)} |f(y)|^p\,\chi_N(y)\,dy
$$
with
$$
\chi_N:= \chi_{\Rn \setminus B(0,N)}\,, \quad N\in\N.
$$
These truncations make sense, for example, for functions $f$ in the uniform Lebesgue space $\mathcal{L}^p$ for $1\leq p < \infty$.
\begin{definition}\label{def:vast}
For $0\leq \lambda < n$ and $1\leq p < \infty$, we define $\Vast$ as the set of all functions $f\in \Lh$ having the vanishing property
\begin{equation}\label{vast-property}
\lim_{N\to \infty} \mathcal{A}_{N,p}(f) =0. \tag{$V^\ast$}
\end{equation}
\end{definition}

By the Lebesgue dominated convergence theorem, we can see that every $L^p$ function has property ($V^\ast$) and hence $V^{(\ast)}L^{p,0} = L^p$. In a sense property ($V^\ast$) allow us to overcome difficulties in applying such a theorem when the functions are not necessarily in $L^p$.

\begin{lemma}\label{lem:vast}
A function $f\in\Lh$ satisfies property ($V^\ast$) if and only if
\begin{equation}\label{vast-uniformcond}
\lim_{N\to \infty}\, \sup_{x\in\Rn} \int_{B(x,r)} |f(y)|^p\,\chi_N(y)\,dy = 0
\end{equation}
uniformly in $r\in]0,R_0]$ for any fixed $R_0>0$.
\end{lemma}

\begin{proof}
It is clear that condition \eqref{vast-uniformcond} immediately implies property ($V^\ast$). Conversely, suppose that $f$ satisfies ($V^\ast$).
Let $R_0>0$ be arbitrary fixed and let $x\in\Rn$. Then there exist $K_0\in\N$ (depending only on $R_0$ and $n$) and $x_j\in B(x,R_0)$, $j=1,\ldots,K_0$, such that
$$B(x,r) \subset \bigcup_{j=1}^{K_0} B(x_j,1)$$
for any $r\in]0,R_0]$. Since
\begin{equation*}
\int_{B(x,r)} |f(y)|^p\,\chi_N(y)\,dy  \leq  \sum_{j=1}^{K_0} \int_{B(x_j,1)} |f(y)|^p\,\chi_N(y)\,dy \leq c\, \sup_{z\in\Rn} \, \int_{B(z,1)} |f(y)|^p\,\chi_N(y)\,dy
\end{equation*}
with $c>0$ not depending on $x$, $r$ and $N$, we get the estimate
$$
\sup_{x\in\Rn} \int_{B(x,r)} |f(y)|^p\,\chi_N(y)\,dy \leq c\, \mathcal{A}_{N,p}(f),
$$
from which \eqref{vast-uniformcond} follows, uniformly in $r\in]0,R_0]$.
\end{proof}


\begin{definition}
For $0\leq \lambda<n$ and $1 \leq p < \infty$, one defines the subset $\Vziast$ of $\Lh$ as
\begin{equation*}\label{vnotation}
\Vziast = \Vz\cap\Vi\cap\Vast.
\end{equation*}
\end{definition}

\begin{example}\label{ex:phi}
Consider the function
$$\varphi(x)= \sum_{k=2}^\infty \chi_{B_k}(x)$$ where $B_k =B(2^k\textrm{e}_1,1)$, $k\in\N$, with $\textrm{e}_1=(1,0, \ldots,0)$. Then
$$ \varphi\in \Vz \cap \Vi \ \ \ \ \text{but} \ \ \ \ \varphi \notin \Vast$$
(cf. proof of Theorem \ref{theo:phi}).
\end{example}

\begin{lemma}\label{lem:closed}
All the vanishing subsets $\Vz$, $\Vi$ and $\Vast$ (and consequently $\Vziast$) are closed in $\Lh$.
\end{lemma}

\begin{proof}
The proof uses standard arguments which can be adapted to all the cases. By this reason we give details for the subset $\Vast$ only.

Let $(f_k)_k$ be a sequence of functions in $\Vast$ converging to $f$ in $\Lh$. We want to show that $f\in \Vast$. For any $\varepsilon>0$ there exists $\overline{k}\in\N$ such that
$$
\mathcal{A}_{N,p}(f-f_{\overline{k}}) \leq \|f-f_{\overline{k}}\|^p_{p,\lambda} < \frac{\varepsilon}{2^{p+1}}.
$$
On the other hand, for such fixed $\overline{k}$, by the vanishing property ($V^\ast$) of $f_{\overline{k}}$ there exists $\overline{N}\in\N$ such that
$$
\mathcal{A}_{N,p}(f_{\overline{k}}) < \frac{\varepsilon}{2^{p+1}}
$$
for every $N\ge \overline{N}$. Hence, we have
$$
\mathcal{A}_{N,p}(f) \leq 2^p\, \mathcal{A}_{N,p}(f-f_{\overline{k}}) + 2^p\, \mathcal{A}_{N,p}(f_{\overline{k}}) <\varepsilon
$$
for arbitrary large $N\in\N$, which completes de proof.
\end{proof}

We end this section with a discussion on the preservation of the vanishing properties introduced above by convolution operators with integrable kernels.

\begin{theorem}\label{theo:invariant}
Let $0\leq \lambda \leq n$, $1\leq p < \infty$. Then the Morrey subspaces $\Vz$, $\Vi$ and $\Vast$ are invariant with respect to convolutions with integrable kernels.
\end{theorem}

\begin{proof}
Let $\phi\in L^1$. By \eqref{Young} it is clear that $f\ast\phi\in \Lh$ when $f\in\Lh$. Let $x\in\Rn$ and $r>0$. By Minkowski's integral inequality and a change of variables, we get
\begin{eqnarray*}
\frac{1}{r^\lambda}\int_{B(x,r)} |(f\ast \phi) (y)|^p \,dy & \leq & \frac{1}{r^\lambda} \left(\int_{\Rn} \left(\int_{B(x,r)}  |f(y-z)\,\phi(z)|^p\,dy\right)^{1/p} \,dz\right)^p\\
& = & \frac{1}{r^\lambda} \left(\int_{\Rn} |\phi(z)| \left(\int_{B(x-z,r)}  |f(u)|^p\,du\right)^{1/p} \,dz\right)^p\\
& \leq & \|\phi\|_1^p \, \sup_{v\in\Rn}\, \mathfrak{M}_{p,\lambda}(f;v,r).
\end{eqnarray*}
From this it follows that
$$f\in \Vz \; \Longrightarrow \; f\ast\phi\in \Vz \ \ \ \ \ \text{and} \ \ \ \ \ f\in \Vi \; \Longrightarrow \; f\ast\phi\in \Vi.$$

It remains to prove that the convolution also preserves property ($V^\ast$). As above, for $x\in\Rn$ and $N\in\N$, we get
$$
\left(\int_{B(x,1)} |(f\ast \phi) (y)|^p \,\chi_N(y)\,dy \right)^{1/p} \leq  \int_{\Rn} |\phi(z)| \left(\int_{B(x,1)}  |f(y-z)|^p\,\chi_N(y)\,dy\right)^{1/p} dz.
$$
We extend the notation from Definition~\ref{def:vast} as follows:
\begin{equation}\label{chi-notation}
\chi_a:= \chi_{\Rn \setminus B(0,a)} \ \ \text{if} \ \ a>0 \ \ \ \ \text{and} \ \ \ \ \chi_a\equiv 1 \ \ \text{if} \ \ a\leq 0.
\end{equation}
So we have $\chi_N(y) \leq \chi_{N-|z|}(y-z)$. Using this fact and changing variables on the right-hand side of the inequality above, we obtain the estimate
$$
\left(\int_{B(x,1)} |(f\ast \phi) (y)|^p \,\chi_N(y)\,dy \right)^{1/p} \leq  \int_{\Rn} |\phi(z)| \left(\sup_{v\in\Rn} \int_{B(v,1)}  |f(u)|^p\,\chi_{N-|z|}(u)\,du\right)^{1/p} dz.
$$
Since this is uniform with respect to $x\in\Rn$, we get
\begin{equation*}\label{key-estimate}
\big[\mathcal{A}_{N,p}(f\ast\phi)\big]^{1/p} \leq \int_{\Rn}|\phi(z)|\,\big[\mathcal{A}_{N-|z|,p}(f)\big]^{1/p} \, dz,
\end{equation*}
with the interpretation
$$
\mathcal{A}_{a,p}(f):=\sup_{x\in\Rn} \int_{B(x,1)} |f(y)|^p\,\chi_a(y)\,dy \,, \ \ \ \ \ a\in\R
$$
(according to \eqref{chi-notation}). Therefore we have
$$\mathcal{A}_{N,p}(f\ast\phi) \rightarrow 0  \ \ \ \text{as} \ \ \ N\to\infty,$$
by the Lebesgue dominated convergence theorem.
\end{proof}

\begin{remark}
An inspection of the last part of the previous proof shows that the convolution operator (with integrable kernel) preserves property ($V^\ast$) of functions belonging to the uniform Lebesgue space $\mathcal{L}^p$ (recall \eqref{Lp-uniform}).
\end{remark}

The preservation of the vanishing properties above is also true for other operators from harmonic analysis. We shall discuss such topic in another paper.


\section{On strict embeddings between Morrey subspaces}\label{sec:strictembed}

\begin{theorem}\label{theo:phi}
For any $0<\lambda<n$ and $1\leq p < \infty$, there are functions in $\Vz\cap\Vi$ which do not have property ($V^\ast$).
\end{theorem}

\begin{proof}
We take the function $\varphi$ from Example \ref{ex:phi},
$$\varphi(x)= \sum_{k=2}^\infty \chi_{B_k}(x)\,,$$
keepping the notation $B_k =B(2^k\textrm{e}_1,1)$. For every $N\in\N$, we have
\begin{eqnarray*}
\mathcal{A}_{N,p}(\varphi) & = & \sup_{x\in\Rn} \sum_{k=2}^\infty \int_{B(x,1)} \chi_{B_k}(y)\,\chi_N(y)\,dy \\
& \geq & \int_{B(2^{\bar{k}}\textrm{e}_1,1)} \chi_{B_{\bar{k}}}(y)\,\chi_N(y)\,dy \\
& = & |B(2^{\bar{k}}\textrm{e}_1,1)| = |B(0,1)|,
\end{eqnarray*}
where $\bar{k}\geq 2$ is any integer chosen such that $2^{\bar{k}}>N+1$. Hence $\varphi$ fails to have property ($V^\ast$).
It remains to show that $\varphi\in \Lh$ and that $\varphi$ has both the vanishing properties ($V_0$) and ($V_\infty$).

For any $x\in\Rn$ and $r\leq 1$, there exists at most one ball $B_{k_0}$ intersecting $B(x,r)$. Then
$$
\mathfrak{M}_{p,\lambda}(\varphi;x,r) =r^{-\lambda} \, \int_{B(x,r)} \chi_{B_{k_0}}(y)\,dy \lesssim r^{n-\lambda},
$$
with the implicit constant not depending on $x$ and $r$. Therefore, noting also that $\lambda\in(0,n)$,
$$\sup_{x\in\Rn}\, \mathfrak{M}_{p,\lambda}(\varphi;x,r) \longrightarrow 0 \ \ \ \text{as} \ \ \ \ r\to 0,$$
which shows that $\varphi$ has the vanishing property at the origin.

Suppose now that $r>1$. We need to count the number of balls intersecting $B(x,r)$. In the case $|x|\leq 2r$ we calculate
$$r+1 \geq |x-2^ke_1|\geq |2^ke_1|-|x| \geq 2^k-2r,$$
so that
$$2^k\leq 3r+1.$$
Therefore, for $|x|\leq 2r$, we get the estimate
\begin{equation}\label{estimate1}
\mathfrak{M}_{p,\lambda}(\varphi;x,r) \leq r^{-\lambda} \, \sum_{k=2}^{[\log(3r)]+1} \left|B(x,r) \cap B_k\right| \leq \frac{\log(3r)}{r^\lambda},
\end{equation}
where $[a]$ stands for the integer part of $a\in\R$.
As before, we should have
$$|x|-r-1\leq 2^k\leq |x|+r+1.$$
Hence, in the case $|x|>2r$, the number of balls $B_k$ intersecting $B(x,r)$ can be estimated by
$$\log(|x|+r+1)-\log(|x|-r)= \log\frac{|x|+r+1}{|x|-r} = \log\left(1+\frac{2r+1}{|x|-r}\right) \leq 2.$$
We then obtain
\begin{equation}\label{estimate2}
\mathfrak{M}_{p,\lambda}(\varphi;x,r) \leq \frac{2}{r^\lambda}.
\end{equation}
Taking into account \eqref{estimate1} and \eqref{estimate2}, we conclude that
$$\sup_{x\in\Rn}\, \mathfrak{M}_{p,\lambda}(\varphi;x,r) \leq \frac{\log(4r)}{r^\lambda} \longrightarrow 0 \ \ \ \text{as} \ \ \ \ r\to \infty,$$
from which the vanishing property at infinity follows (note that $\lambda>0$).
\end{proof}

\begin{corollary}
For any $0<\lambda < n$ and $1\leq p <\infty$ we have
$$
\Vziast \ \subsetneqq \ \Vz\cap\Vi \ \subsetneqq \ \Vz \ \subsetneqq \ \Lh.
$$
\end{corollary}

The next result (and its corollary) shows that the new subspace $\Vziast$ is strictly smaller than Zorko subspace $\Zor$ (from \eqref{zorko-subs}). Its proof  uses the fact that Morrey functions having all the vanishing properties above can always be approximated by compactly supported functions. However, for convenience, this claim will be proved later in Section~\ref{sec:approx-subspaces}.

\begin{theorem}
For any $0\leq\lambda < n$ and $1\leq p < \infty$ we have $\Vziast \subset \Zor$.
\end{theorem}

\begin{proof}
Let $f\in\Vziast$. For any $\varepsilon>0$ there exists $g\in \Lh$ with compact support such that
$$
\|f-g\|_{p,\lambda} < \varepsilon/4
$$
(cf. Step~2 in the proof of Theorem \ref{theo:approx-main}). Since for any $\xi\in\Rn$ we have
$$\|\tau_\xi f-f\|_{p,\lambda} \leq 2\, \|f-g\|_{p,\lambda} + \|\tau_\xi g-g\|_{p,\lambda}$$
it suffices to show that the second norm is less that $\varepsilon/2$ for small values of $|\xi|$. By the ($V_0$) and ($V_\infty$) properties, there exist $r_0, R>0$ such that
$$
S_1:= \sup_{x\in\Rn\atop 0<r<r_0} \mathfrak{M}(\tau_\xi g-g; x,r) \leq 2^p \sup_{x\in\Rn\atop 0<r<r_0} \mathfrak{M}(g; x,r) < (\varepsilon/2)^p
$$
and
$$
S_2:= \sup_{x\in\Rn\atop r>R} \mathfrak{M}(\tau_\xi g-g; x,r) \leq 2^p \sup_{x\in\Rn\atop r>R} \mathfrak{M}(g; x,r) < (\varepsilon/2)^p.
$$
For such fixed $r_0$ and $R$, we estimate
$$
\|\tau_\xi g-g\|^p_{p,\lambda} \leq \max\{S_1, S_2, S_3\}
$$
with
$$
S_3:= \sup_{x\in\Rn\atop r_0\leq r\leq R} \mathfrak{M}(\tau_\xi g-g; x,r).
$$
We have
$$
S_3 \lesssim \sup_{x\in\Rn} \int_{B(x,R)}|g(y-\xi)-g(y)|^p\,dy \leq \max\left\{ \sup_{|x|<M} (\cdots)\,, \sup_{|x|>M} (\cdots) \right\}.
$$
where $M>0$ is chosen below. Since $g$ has compact support there exists $K>0$ such that $g(u)=0$ if $|u|>K$. In the case $|x|>M$ we have
$$
\int_{|y-x|<R}|g(y-\xi)-g(y)|^p\,dy = \int_{|z|<R}|g(z+x-\xi)-g(z+x)|^p\,dz.
$$
Hence, if we choose $M>R+K+1$ then $g(z+x)=0$ and $g(z+x-\xi)=0$ for small values of $|\xi|$, say $|\xi|<1$, since
$$|z+x|>M-R>K \ \ \ \ \text{and} \ \ \ \  |z+x-\xi|\geq |z+x|-|\xi|>K.$$
Let then $M>R+K+1$ be fixed and let us now estimate the integral when $|x|<M$. In this case we are just taking the $L^p$-norm on a ball centered at the origin with fixed radius, precisely $B(0,R+M+1)$, again for $|\xi|<1$.  Therefore we also obtain
$$S_3 < (\varepsilon/2)^p$$
by the continuity of the $L^p$-norm with respect to translations.
\end{proof}

\begin{remark}
From the previous proof, we see that for any $\varepsilon>0$ and $f\in \Vz\cap \Vi$, there are $r_0,R>0$ such that
$$
\|\tau_\xi f-f\|^p_{p,\lambda} \leq \max\{\varepsilon, S_{r_0,R}(\xi)\}
$$
for all $\xi\in\Rn$, where
$$S_{r_0,R}(\xi)=r_0^{-\lambda}\,\sup_{x\in\Rn} \int_{B(x,R)} |f(y-\xi)-f(y)|^p\,dy.$$
 Thus, for Morrey functions in the intersection $\Vz\cap \Vi$, the Zorko property is reduced to the continuity of the translation operator in the uniform Lebesgue space $\mathcal{L}^p$. Note also that the vanishing property ($V^\ast$) ensures the continuity of the translation operator in $\mathcal{L}^p$, and consequently Zorko property.
\end{remark}

\begin{corollary}
For any $0<\lambda < n$ and $1\leq p <\infty$ there holds
$$
\Vziast \ \subsetneqq \ \Zor \ \subsetneqq \ \Vz \ \subsetneqq \ \Lh.
$$
\end{corollary}

\section{Approximation in Morrey subspaces}\label{sec:approx-subspaces}

As observed by Zorko \cite{Zor86} there are functions in $\Lh$ that cannot be approximated even by continuous functions. It is the case of functions with the form  $f_{x_0}(x)=|x-x_0|^{\frac{\lambda-n}{p}}$ for fixed $x_0\in\Rn$. It should be mentioned that the Morrey norm in \cite{Zor86} was taken as the sum $\|\cdot\|_p + \|\cdot \|_{p,\lambda}$. In particular, Morrey functions there were supposed to be in $L^p$ which is not necessarily the case in the approach followed in our paper. Anyway, it was the failure in approximation by nice functions that motivated the introduction of the subspace $\Zor$.

In the intersection $\Lh \cap L^p$ standard approximations using mollifiers can be used in order to approximate functions with Zorko property by $C_0^\infty$ functions (cf. \cite[Theorem~2.3]{Adam15}). However, things change when we deal with the Morrey space $\Lh$, $\lambda>0$,  with the norm defined by \eqref{HomMorreynorm}, since then we cannot approximate all Morrey functions by regular compactly supported functions.

Consider the usual dilations $\phi_t(x)= t^{-n}\phi(x/t)$, $t>0$, where $\phi$ is an integrable function with $\|\phi\|_1=1$. Suppose that $f\in\Zor$. Standard calculations and Minkowski's integral inequality yield
\begin{equation*}
\|f\ast \phi_t-f\|_{p,\lambda} \leq \int_{\Rn} \|\tau_{t z}f - f\|_{p,\lambda} \,|\phi(z)|\,dz.
\end{equation*}
Since $\|\tau_{t z}f - f\|_{p,\lambda}$ as $t\to 0$ and $\|\tau_{t z}f - f\|_{p,\lambda} \leq 2  \|f\|_{p,\lambda}$ for any $t>0$ and $z\in\Rn$, by the Lebesgue theorem we get
\begin{equation}\label{approx}
 \|f\ast \phi_t - f\|_{p,\lambda} \rightarrow 0 \ \ \ \ \text{as} \ \ \ \ t\to 0.
\end{equation}
If we take smooth kernels, say $\phi\in\cS$, then the mollifiers $f\ast\phi_t \in \Zor\cap C^\infty$ for any $t>0$. Consequently, we derive the following result:
\begin{theorem}
Let $0\leq \lambda < n$ and $1\leq p < \infty$. Then every Morrey function with Zorko property can be approximated in Morrey norm by $C^\infty$ functions. Moreover, we have $$\overline{\Zor\cap C^\infty} = \Zor.$$
\end{theorem}

We discuss now the approximation of Morrey functions having both vanishing properties at the origin and at infinity.

\begin{theorem}\label{uniform-continuous}
Let $0\leq \lambda < n$ and $1\leq p < \infty$. If $f\in\Vz\cap\Vi$ is uniformly continuous then $f$ can be approximated in Morrey norm by functions from $\Vz\cap\Vi \cap C^\infty$.
\end{theorem}

\begin{proof}
Let $\phi\in\cS$ with $\|\phi\|_1=1$. Then $f\ast\phi_t \in \Vz\cap\Vi\cap C^\infty$ for any $t>0$ by Theorem~\ref{theo:invariant}. Thus it remains to show that $f\ast \phi_t \rightarrow f$ in $\Lh$ as $t \to 0$.

Let $\varepsilon>0$. For any $x\in\Rn$, $r>0$ and $t>0$, we have
$$
\left(\int_{B(x,r)} |(f\ast\phi_t)(y)-f(y)|^p\,dy \right)^{1/p} \leq \int_{\Rn} |\phi_t(z)| \left(\int_{B(x,r)} |f(y-z)-f(y)|^p\,dy \right)^{1/p} dz.
$$
Since $f\in\Vz\cap\Vi$, there are $r_0, R>0$ such that
$$
\frac{1}{r^\lambda} \int_{B(x,r)} |f(y-z)-f(y)|^p\,dy < \varepsilon
$$
for every $r<r_0$ or $r>R$ (and all $x,z\in\Rn$). Thus we have
\begin{equation}\label{aux1}
\sup_{r>0}\frac{1}{r^\lambda} \int_{B(x,r)} |(f\ast\phi_t)(y)-f(y)|^p\,dy \leq \max\{ \varepsilon, S_{r_0,R}(x,t)\}
\end{equation}
where
$$
S_{r_0,R}(x,t) := \int_{\Rn} |\phi_t(z)| \left(\frac{1}{r_0^\lambda}\int_{B(x,R)} |f(y-z)-f(y)|^p\,dy \right)^{1/p} dz.
$$
By the uniform continuity of $f$ one can find $\delta>0$ such that $|f(y-z)-f(y)| < \varepsilon$ for any $y$ and $z$ with $|z|<\delta$. For such fixed $\delta$ we split the outer integral above into
\begin{equation}\label{split}
S_{r_0,R}(x,t) = \int_{|z|<\delta} (\cdots)\,dz + \int_{|z|\geq\delta} (\cdots)\,dz.
\end{equation}
For the first integral we use the uniform continuity of $f$ and get
\begin{equation}\label{aux2}
\int_{|z|<\delta} (\cdots)\,dz \leq \varepsilon \, r_0^{-\lambda}\, |B(x,R)| \,\int_{|z|<\delta} |\phi_t(z)|\,dz \leq |B(0,1)|\, r_0^{-\lambda}\, R^n\,\varepsilon
\end{equation}
for every $x\in\Rn$ and $t>0$. In the second integral in \eqref{split}, we use the fact that $f\in\Lh$ and $\phi\in\cS$ to derive the inequality
\begin{equation}\label{aux3}
\int_{|z|\geq \delta} (\cdots)\,dz \leq c\,r_0^{-\lambda}\, R^\lambda\,t \,\|f\|_{p,\lambda}
\end{equation}
with $c>0$ depending only on $\phi$, $n$ and $\delta$. In particular, we used the estimate
$$
\int_{|z|\geq \delta} |\phi_t(z)|\,dz \lesssim \int_{|z|\geq \delta} \frac{t}{|z|^{n+1}}\,dz \lesssim t\,,
$$
where the implicit constant is independent of $t$. Using \eqref{aux2} and \eqref{aux3} in \eqref{split}, from \eqref{aux1} we obtain
$$
\sup_{x\in\Rn\atop r>0}\mathfrak{M}_{p,\lambda}(f\ast\phi_t-f;x,r) \lesssim \varepsilon
$$
for sufficiently small $t>0$. This implies $\|f\ast \phi_t - f\|_{p,\lambda} \rightarrow 0$ as $t\to 0$, and hence the proof is complete.
\end{proof}

Finally we discuss the approximation of Morrey functions having all the vanishing properties.


%

\begin{theorem}\label{theo:approx-main}
Let $0\leq \lambda <n$ and $1\leq p<\infty$. Then every function in $\Vziast$ can be approximated in Morrey norm by $C_0^\infty$ functions.
\end{theorem}

\begin{proof}
%
%

\underline{Step 1}: The claim holds true for functions $f\in\Vziast$ with compact support. In fact, if the kernel $\phi\in C_0^\infty$ then the mollifiers $f\ast\phi_t$  have compact support and belong to $\Vziast$ (cf. \ref{theo:invariant}). Moreover, they approximate $f$ in Morrey norm (recall the discussion in the beginning of this section leading to \eqref{approx}).

\underline{Step 2}: We show now that functions from $\Vziast$ can be approximated by compactly supported functions in Morrey norm. Let $f\in\Vziast$. As before let $\chi_k := \chi_{\Rn\setminus B(0,k)}$, $k\in\N$. For each $k$, set
$$ f_k=f \ \ \ \text{on the ball} \ B(0,k) \ \ \ \ \text{and} \ \ f_k=0\ \ \text{otherwise}.$$

Let $\varepsilon>0$ be arbitrary. By the vanishing properties ($V_0$) and ($V^\infty$) one finds $r_0, R>0$ such that
$$
\sup_{x\in\Rn} \mathfrak{M}_{p,\lambda}\big(f-f_k;x,r\big) = \sup_{x\in\Rn} \mathfrak{M}_{p,\lambda}\big(f\chi_k;x,r\big) < \varepsilon
$$
for all $k\in\N$ and every $r<r_0$ or $r>R$. Hence
$$
\|f-f_k\|^p_{p,\lambda} < \max\{\varepsilon,S_{r_0,R}\}
$$
where
$$
S_{r_0,R}(k):= \sup_{x\in\Rn \atop r\in[r_0,R]} \mathfrak{M}_{p,\lambda}\big(f\chi_k;x,r\big).
$$
Now, by the vanishing property ($V^\ast$) and Lemma \ref{lem:vast}, we get
$$
S_{r_0,R}(k)\leq \sup_{x\in\Rn} r_0^{-\lambda}\int_{B(x,R)} |f(y)|^p\,\chi_k(y)\,dy < \varepsilon
$$
for all $k$ large enough. Therefore,
$$
\|f-f_k\|_{p,\lambda}  \rightarrow 0 \ \ \ \text{as} \ \ \ k\to\infty.
$$
\end{proof}

From the previous result and the fact that $\Vziast$ is closed (cf. Lemma~\ref{lem:closed}), we obtain an explicit description of $\overset{\circ}{L}{}^{p,\lambda}$, the closure of $C_0^\infty$ in the homogeneous Morrey spaces.

\begin{corollary}
For $0\leq \lambda <n$ and $1\leq p<\infty$ the set $C_0^\infty$ is dense in $\Vziast$. Moreover, $\overset{\circ}{L}{}^{p,\lambda}=\Vziast$.
\end{corollary}

Recalling the known result \eqref{bidual} we see that $\Lh$ is precisely the bidual of the new subspace $\Vziast$.

\section{A generalization of embeddings into weighted $L^p$ spaces}\label{sec:improved-embed}

The next result generalizes some known embeddings of Morrey spaces into weighted Lebesgue spaces (see, for instance, \cite{Kato92}, \cite{RosTri15}).
Recall that a non-negative function $g$ on $[0,\infty)$ is called \emph{almost decreasing} if there exists a constant $c>0$ such that $g(s) \le c\,g(t)$ for all $s\ge t (\ge 0)$.

\begin{theorem}\label{the:weighted-embed}
Let $1\le p<\infty$ and $0<\lambda<n$.
Let $w=w(|x|)$ be a radial weight such that $w(t) \approx 1$ for $0\leq t< 1$. Then
\begin{equation*}
\Lh(\Rn) \hookrightarrow L^p(\Rn,w) \ \ \ \ \text{and} \ \ \ \ \Lnh(\Rn) \hookrightarrow L^p(\Rn,w)
\end{equation*}
if $\,w(t)^p\,t^\gamma\,$ is almost decreasing, and
\begin{equation}\label{integral}
\int_1^\infty w(t)^p\,t^{\gamma-1}\,dt<\infty
\end{equation}
with $\gamma =\lambda$ in the homogeneous case and $\gamma=n$ in the inhomogeneous case.
\end{theorem}

\begin{proof}
We split the weighted norm of $f$ as
$$
\|f\|_{L^p(w)}^p = \int_{|y|<1} |f(y)|^p\,w(y)^p\,dy + \int_{|y|\geq 1} |f(y)|^p\,w(y)^p\,dy =: I_1(f) + I_2(f).
$$
Since $w(|y|) \approx 1$ for $|y|<1$ we have
$$ I_1(f) \lesssim \|f\|_{\Lh}^p \ \ \ \ \text{and} \ \ \ \ I_1(f) \lesssim \|f\|_{\Lnh}^p.$$

As regards to $I_2(f)$ the estimation is different in each situation. In the homogeneous case we have
\begin{equation*}
I_2(f)  =  \sum_{k=0}^\infty \int_{2^{k} \le |y| < 2^{k+1}} |f(y)|^p\, w(y)^p\,dy \lesssim  \sum_{k=0}^\infty w(2^k)^p \,2^{k\lambda} \, \|f\|_{\Lh}^p \lesssim  \|f\|_{\Lh}^p\,,
\end{equation*}
where in the first inequality we used the monotonicity of $w(t)^p\,t^\gamma$ and in the second we used the integral assumption \eqref{integral} with $\gamma=\lambda$.

In the inhomogeneous case we use similar arguments (with $\gamma=n$ in the integral condition \eqref{integral}) and estimate the integral as follows:
\begin{equation*}
I_2(f)  \leq  \sum_{m\in\Zn\atop |m|\neq 0} \int_{B(m,1)} |f(y)|^p\, w(y)^p\,dy \lesssim  \sum_{|m|\geq 1} w(|m|-1)^p \, \|f\|_{\Lnh}^p \lesssim  \|f\|_{\Lnh}^p.
\end{equation*}
\end{proof}


\begin{example}\label{ex:weight}
It is easy to see that the weights
$$w(x)=(1+|x|)^\alpha \approx (1+|x|^2)^{\alpha/2} =: w_\alpha(x)$$
satisfy the assumptions of Theorem~\ref{the:weighted-embed} when $\alpha<-\gamma/p$. Hence we have
$$\Lh(\Rn) \hookrightarrow L^p(\Rn,w_\alpha) \ \ \ \ \ \text{if} \ \ \ \alpha<-\lambda/p$$
and
$$\Lnh(\Rn) \hookrightarrow L^p(\Rn,w_\alpha) \ \ \ \ \ \text{if} \ \ \ \alpha<-n/p$$
(cf. \cite[p.132]{Kato92} and \cite[Theorem~3.1]{RosTri15}, respectively).
\end{example}

\begin{example}
The weights
$$
w(x)=(1+|x|)^{-\frac{\gamma}{p}}\,  \ln^{-\beta}(e+|x|), \ \ \ \ \textrm{with} \ \ \ \beta>1/p,
$$
are also admissible in the sense of Theorem~\ref{the:weighted-embed}. This corresponds to the limiting case $\alpha=-\gamma/p$ given in Example~\ref{ex:weight}.

\end{example}


\end{document}